\documentclass{article}

\usepackage{times, amsmath, amsfonts, amssymb, amsthm}
\usepackage{times}
\usepackage{natbib}
\usepackage{setspace}

\newcommand{\dd}{\,\text{\rm d}}             
\newcommand{\ee}{\,\text{\rm e}}

\newcommand{\cL}{\mathcal{L}}

\newcommand{\1}{{\bf{1}}}

\theoremstyle{plain}
\newtheorem{theorem}{Theorem}[section]
\newtheorem{lemma}[theorem]{Lemma}

\newtheorem{corollary}[theorem]{Corollary}

\theoremstyle{definition}
\newtheorem{remark}[theorem]{Remark}

\newtheorem{example}[theorem]{Example}

\newtheorem{lemmadefi}[theorem]{Lemma and Definition}

\renewenvironment{proof}[1][] {\noindent {\bf Proof#1.} }{\hspace*{\fill}$\square$\medskip\par}

\newcommand{\ve}{{\varepsilon}}

\newcommand{\N}{{\mathbf N}}

\newcommand{\R}{{\mathbf R}}

\newcommand{\E}{{\mathbf E}}
\newcommand{\F}{{\cal F}}

\renewcommand{\P}{{\mathbf P}}

\author{Xue-Mei Li%
  \thanks{Department of Mathematics, 
          The University of Warwick, 
          Coventry CV4 7AL, 
          UK; \ 
        \small{\tt Xue-Mei.Hairer{\scriptsize @}warwick.ac.uk} Research supported by  the EPSRC (EP/E058124/1).}
\and Michael Scheutzow%
  \thanks{Institut f\"ur Mathematik, MA 7-5, Fakult\"at II, 
        Technische Universit\"at Berlin, 
        Stra\ss e des 17. Juni 136, 10623 Berlin, FRG;  \ 
        \small{\tt ms{\scriptsize @}math.tu-berlin.de}}}

\title{Lack of strong completeness for stochastic flows}

\date{\today}

\begin{document}  \maketitle

\begin{abstract}\noindent
It is well-known that a stochastic differential equation (SDE) on a Euclidean space driven 
by a Brownian motion with Lipschitz coefficients generates a stochastic flow of 
homeomorphisms. When the coefficients are only locally Lipschitz, then a maximal continuous flow still exists 
but explosion in finite time may occur. 
If -- in addition -- the coefficients grow at most linearly, then this flow has the property
that for each fixed initial condition $x$, the solution exists for all times almost surely. If the 
exceptional set of measure zero can be chosen independently $x$, then the maximal flow is called 
{\em strongly complete}. The question, whether an SDE with locally Lipschitz continuous coefficients 
satisfying a linear growth condition is strongly complete was open for many years. 
In this paper, we construct a 2-dimensional SDE
with coefficients which are even bounded (and smooth) and which is {\em not} strongly complete thus answering 
the question in the negative. 

  \par\medskip

\end{abstract}


\section{Introduction}\label{intro}
We will assume throughout that $(\Omega,\F,\P)$ is a given probability space. Let us consider the following 
stochastic differential equation (SDE) on $\R^d$
\begin{equation}
\label{sde}
dX_t=\sum_{i=1}^n \sigma_i(X_t) \dd B_t^i+\sigma_0(X_t)\dd t,
\end{equation}
where $B^1,...,B^n$ are independent standard Wiener processes defined on $(\Omega,\F,\P)$ and 
the $\sigma_i$ are locally Lipschitz continuous vector fields and hence the SDE has a unique local solution
for each initial condition $X(0)=x$.

It is well-known that such SDEs
with  global  Lipschitz coefficients  do not only possess a unique global solution for each fixed 
initial condition but also a version of the global solution which is continuous in the initial data,  
\cite{Blagovescenskii-Freidlin}.  
This global  solution generates in fact a {\em stochastic flow of homeomorphisms}
\cite{Kunita-paper}, \cite{Kunita}. 
Furthermore it is  well-known that  for a unique global strong solution to exist, it suffices that 
the coefficients of the SDE satisfy a suitable local regularity condition and a growth condition at infinity -- for example 
a local Lipschitz condition and a linear growth condition.
Local Lipschitz continuity guarantees local existence and uniqueness of solutions as well as continuous dependence of the 
local flow on initial conditions 
while the linear growth condition (which can in fact be weakened a bit by allowing additional logarithmic terms) 
allows us to pass from local to global by a Gronwall's lemma procedure. Both conditions are almost necessary as the lack of 
a local Lipschitz condition can lead to lack of pathwise uniqueness and the lack of linear growth can lead to explosion.  
SDEs which have a global strong solution for each initial condition are said to be {\em complete} or {\em weakly complete}.  
It is well known that a complete SDE need not have a continuous modification of the solution as a function of 
time and the initial data. 
This marks a departure of  the theory of stochastic flows from that of deterministic ordinary differential equations.  
However there is so far only a pitifully small number of examples of complete stochastic differential equations whose 
solutions do not admit a continuous modification as a function of  time and  initial data.  
Not a single such example has coefficients which are 
locally Lipschitz and of linear growth (in spite of a remark in \cite{Fang-Imkeller-Zhang07}
stating the contrary). The basic example is the following:
\begin{eqnarray*}
\dd x_t&=&(y_t^2-x_t^2)\dd B_t^1-2x_ty_t\dd B_t^2\\
\dd y_t&=&-2x_ty_t \dd B_t^1 + (x_t^2-y_t^2)\dd B_t^2,
\end{eqnarray*}
where $B^1,\,B^2$ are independent standard Brownian motions. 
It was first given by Elworthy \cite{Elworthy82} (see also \cite{Carverhill}, \cite{flow-li} for 
further discussion).  This SDE is equivalent to 
$\dd x_t=\dd W_t$ on $\R^2\backslash \{0\}$ for some 2-dimensional Wiener process $W$ 
through the transformation $z\mapsto {1\over z}$ in the complex plane representation.
It is clear that $x+W_t$ does not explode in $\R^2\backslash\{0\}$ for each individual $x$, as a Brownian motion does 
not see single points. 
The unique maximal flow is given by $\{x+W_t(\omega), x\in \R^2\backslash\{0\}\}$ (up to explosion) 
and it explodes for any given $\omega$.

Our aim here is to construct stochastic differential equations which are complete but not {\em strongly complete}, i.e.\ 
which do not admit a continuous modification. In the examples, the lack of strong completeness 
is achieved by rapidly oscillating vector fields.
The example which we will present in the next section shows that even under the additional constraint that 
the equation has no drift and the diffusion coefficient is bounded and $C^{\infty}$, 
there may not exist a global solution flow. Even more, in our example the SDE is driven 
by a single one-dimensional Brownian motion. Note that such examples are clearly impossible for scalar equations, 
so the dimension of the state space of the SDE has to be at least 2.
Our examples are in $\R^2$.

\section{Negative Results}\label{examples}
Below, we will construct an example of an SDE in the plane of the form
\begin{eqnarray}\label{SDE1}
\begin{split}
\dd X(t)&=\sigma(X(t),Y(t))\dd W(t)\\ 
\dd Y(t)&=0,
\end{split}
\end{eqnarray}
which is {\em not} strongly complete and where $\sigma:\R^2 \to (0,\infty)$ is bounded, bounded away from 0  and $C^{\infty}$. 

Before going into details, let us explain the idea of the construction. From \eqref{SDE1} it is 
clear that in our example trajectories move on straight lines parallel to the first coordinate axis. 
If the equation was driven by a family of Brownian motions (rather than a single one) which are indexed by $y \in \R$ and are 
independent for different values of $y$, then clearly the supremum over all solutions at time 1 (say) with initial 
conditions of the form $(0,y)$, $0 \le y \le 1$ would be infinite. Such a modification would of course contradict 
our assumptions but we can (and will) try to approximate this behavior using an equation 
of type \eqref{SDE1} with carefully chosen $\sigma$ (satisfying all properties stated above). Our $\sigma$ will exhibit increasingly 
heavy oscillations when $x \to \infty$ with different frequencies for different values of $y$. Thus we can make sure that for 
different values of $y$, the solutions behave (for large $x$) almost as if they were driven by independent Brownian motions 
-- in spite of the fact that they are all driven by the same Brownian motion. 
If we manage to construct $\sigma$ such that approximate independence sets in 
sufficiently quickly, then we can hope to observe exploding solutions, i.e.\ lack of strong completeness. 
In fact it will turn out 
that in our example, solutions for different values of $y$ will not be asymptotically independent but that solutions can 
be asymptotically written as a sum of two Brownian motions: one which is the same for all $y$ and another one which is 
independent for different $y$. This property suffices to show that strong completeness does not hold.
\subsection{Preliminaries}

The following lemma which is proved in \cite{Kunita}, Theorem 4.7.1 ensures the existence of a maximal (continuous) flow 
generated by the SDE \eqref{sde}.  
\begin{lemmadefi}[Maximal Flow]
Suppose that the vector fields $\sigma_i$  are locally Lipschitz continuous.  Then there exist a function 
$\tau:\R^d \times \Omega \to (0,\infty]$ and a map 
$\phi: \{(t,x,\omega): x \in \R^d,\,\omega \in \Omega,\,t \in [0,\tau(x,\omega))\}\to \R^d$ 
such that the following holds:
\begin{enumerate}
\item  For each $x \in \R^d$, $\phi_t (x,.)$ solves \eqref{sde} with initial condition $x$ on $[0,\tau(x,\omega))$,
\item $\phi_t(x, \omega):   \{(t,x): t<\tau(x,\omega)\}\to \R^d$ is a continuous function of $(t,x)$; 
\item for each $x$, $\limsup_{t \to \tau (x,\omega)}|\phi_t(x, \omega)| = \infty$ on  $\{\omega: \tau(x, \omega)<\infty\}$.
\end{enumerate}
The map $\phi$ is called a {\em maximal flow}. $(\phi,\tau)$ are unique up to a null set. 
If, for each $x \in \R^d$, we have $\tau(x,\omega)=\infty$ almost surely, then we call the 
SDE (or the maximal flow) {\em complete} or {\em weakly complete}. If, moreover, there exists a set $\Omega_0$ such that $\tau(x,\omega)=\infty$ 
for all $x \in \R^d$ and all $\omega \in \Omega_0$, then the SDE or the maximal flow are called {\em strongly complete}.
\end{lemmadefi}
Usually, flows are assumed to have two time parameters (an additional one for the starting time) and to satisfy a corresponding 
composition property, but in this paper we will not dwell on this.

We review briefly some positive results for strong completeness of the SDE (\ref{sde})  
in terms of growth conditions on the coefficients of the SDE.  
For simplicity assume that   the vector fields are $C^2$ and consider the derivative 
equation:
\begin{equation}
\dd v_t=\sum_i D\sigma_i(x_t)(v_t)\dd B_t^i+D\sigma_0(x_t)(v_t)\dd t.
\end{equation}
If $T_x\phi_t(v)$ is the solution to the derivative equation with initial value $v$, 
the SDE is strongly complete if it does not explode starting from one starting point and if for some $p>n-2$, 
$\sup_{x\in K}\E \sup_{s\le t} |T_t\phi_s|^p1_{s<\tau(x)}$ is finite for every compact set $K$  \cite{flow-li}. 
For $n=2$, it is sufficient to take $p=1$. In terms of the vector fields $\sigma_i$ there is the following 
theorem summarized from  Theorem 5.1, which is valid for SDEs on manifolds, and Lemma 6.1  in  \cite{flow-li}. 
\begin{theorem}
Let $\cL$ be the generator of the SDE \eqref{sde} and $\sigma^\cL $ its symbol so 
$\sigma^\cL (dg, dg)={1\over 2} \cL(g^2)-g\cL g$.
If $g$ is a Lyapunov function in  the sense that ${\mathcal L} g+{1\over 2}\sigma^\cL(dg,dg)\le c$ 
and $\lim_{x\to \infty} g(x)=\infty$ then the SDE is strongly complete if the solution 
from some initial point exists globally  and  if  $|D\sigma_i(x)|^2 \le g(x)$ and 
$2{\langle D \sigma_0(x)(v),v\rangle }\le   g(x)|v|^2$ for all $v\in \R^d$ .
\end{theorem}
Examples of such Lyapunov functions include $g(x)=1+\ln (1+|x|^2)$ and $g(x)=x^\ve$.  For a recent result on 
strong completeness, see \cite{Fang-Imkeller-Zhang07}. 
For earlier works, see  also \cite{Baxendale} and \cite{Carverhill-Elworthy}.
For results on strong completeness for stochastic delay differential equations, the reader is referred to \cite{MS03}.

Let us explain the completeness and strong completeness concepts using stopping times. 
First note the following observation.  
Let $U_1\subset U_2\subset U_3\subset ... $ be an exhausting sequence of bounded open subsets of $\R^d$. 
Let $\tau_n(x)$ be the first exit time of the solution, starting from a point $x$  from $U_n$.   
If there exists a non-increasing sequence of  $\delta_n$ such that $\sum \delta_n=Ò\infty$ and 
$\P\{\tau_n(x)\le t\}\le ct^2$ for any  $ t\le \delta_n$ and $x\in U_{n-1}$,
then an application of the Borel-Cantelli Lemma shows that weak completeness holds \cite{infinity-li}. 
We state the corresponding elementary lemma for strong completeness with  converse whose essence will be used  
in the proof for the claim in the example we will construct.

\begin{lemma}\label{stopping times}
Take $\phi_t(x,\omega)$ to be the maximal flow and let  $K$ be a compact set and $\tau^K_n:=\inf\{t>0:\phi_t(K) \nsubseteq U_n\}$. 
Define  $\tau^K=\inf_{x\in K} \tau(x)$.
If  for  two sequences $\{a_n\}$ and $\{b_n\}$ with $\sum_{n}a_n=\infty$ and $\sum b_n<\infty$,
$$P\{ \tau_n^K-\tau_{n-1}^K \le a_n, \tau_{n-1}^K<\infty\} \le b_n,$$
then $\tau^K$ is infinite and if this property holds for every compact set $K$, then we have strong completeness.  
 
Conversely, let $\{a_n\}$ and $\{b_n\}$ be two summable sequences.  Let $T_j$ be finite random  times such that
 $\tau^K\le \sum_j T_j$, then 
 $\tau^K<\infty$ almost surely if 
 $$\P\{ T_n\ge a_n\} \le b_n.$$
\end{lemma}

\subsection{A Bunch of Lemmas}

Lemma \ref{homo} below is the key to the construction of our example. While known results in 
homogenization theory state convergence in law of the solutions of a sequence of SDEs like \eqref{homogen} 
to a Brownian motion (with a certain {\em effective} 
diffusion constant) we are not aware that the asymptotics of the joint laws of the solutions has been 
investigated in the  literature. The proof of Lemma \ref{homo} will use the following lemma.

\begin{lemma}\label{convergence}
Let $X^{\ve}=\big( X_1^{\ve}, X_2^{\ve},...\big),\,\ve>0$ be a family of continuous local martingales 
starting at 0. Let 
$B_1,B_2,...$ be independent standard Brownian motions, $\alpha_{ij}\in \R,\,i,j \in \N$ such that 
$\sum_j \alpha_{ij}^2 <\infty$ for all $i \in \N$, 
$V_i:=\sum_{j=1}^{\infty} \alpha_{ij}B_j,\,i \in \N$, and $V=(V_1,V_2,...)$. 
If the quadratic variation $[X_k^{\ve},X_l^{\ve}]_t$ converges in law 
to $[V_k,V_l]_t=t\sum_{j=1}^{\infty} \alpha_{kj}\alpha_{kl}$ for all $k,l \in \N,\, t \ge 0$, 
then $X^{\ve}$ converges to $V$ weakly as $\ve \to 0$.
\end{lemma}

\begin{proof} 
This follows from Theorem VIII.2.17 in \cite{JS} (the theorem is formulated for a family of $\R^n$-valued 
$X^{\ve}$ rather than  sequences but the statement for sequences is an immediate corollary). 
See also Revuz-Yor \cite{Revuz-Yor}.
\end{proof}

\begin{lemma}\label{homo} Let $H_i:{\bf R} \to [0,\infty)$, $i=1,2$ be Lipschitz continuous with 
period 1 and assume that $H_1$ is non-constant and $H_1(x)+H_2(x)>0$ for all $x$.
Let $W_i$, $i=1,2$ be independent standard one-dimensional Brownian motions and $\ve>0$. Consider the SDE
\begin{equation}\label{homogen}
\begin{split}
\dd X^\ve (t)&=H_1\big(\frac{1}{\ve} X^\ve (t)\big) \dd W_1(t)+ H_2\big( \frac{1}{\ve} X^\ve (t)\big) \dd  W_2(t) \\
\;X^\ve (0)&=x.
\end{split}
\end{equation}
There exist  $\hat \alpha,\hat \beta>0$ (not depending on the initial condition $x$) such that the following holds: 
if $(\ve_n)$ is a sequence of 
positive reals satisfying $\ve_{n+1}/\ve_n \to 0$ as $n \to \infty$, then $(X^{\ve_n}-x,X^{\ve_{n+1}}-x,...)$ 
converges weakly to $(\hat \alpha B_0 + \hat \beta B_1,\hat \alpha B_0 + \hat\beta B_2,...)$
as $n \to \infty$, where $B_0,B_1, \dots$ are independent standard Brownian 
motions. 
\end{lemma}

\begin{proof} By the previous lemma, it suffices to show that there exist $\hat\alpha, \hat\beta >0$ such that 
$[X^\ve-x]_t$ and $[X^\ve-x,X^{\tilde \ve}-x]_t$ 
converge to $(\hat\alpha^2 + \hat\beta^2)t$ respectively $\hat\alpha^2t$ in law for each $t \ge 0$ as $\ve \to 0$ and $\tilde \ve \to 0$ 
such that $\tilde \ve/\ve \to 0$.
Set $z^\ve(t)={1\over  \ve} X^\ve(t\ve^2)$ and  let $W_i^\ve(t)=\frac{1}{\ve}W_i(t\ve^2)$, $i=1,2$, be the rescaled Brownian motions.
Then $z^\ve(t)$ satisfies:
$$\dd z^\ve(t)= H_1\left(z^\ve(t)\right )  \dd  W_1^\ve(t)
+ H_2\left(z^\ve(t)\right )  \dd  W_2^\ve(t).$$
The projection to $[0,1]$  is an ergodic Markov process with invariant measure $\mu$:
$$\mu(\dd y)={1\over v}
{\dd y\over H_1^2(y)+H_2^2(y)} $$
for $v={ \int_0^1 {1\over H_1^2(y)+H_2^2(y)} \dd y}$ the normalising constant.
If $f$ is a continuous periodic function with period 1, denote by  $\bar f$ its average:
$$\bar f= \int_0^1 f(x)\dd\mu(x).$$ Then by the law of large numbers for $z^\ve$, for each fixed $t \ge 0$,
\begin{equation}
\begin{split}
&\lim_{\ve  \to 0}  \int_0^t f\Big(\frac{1}{\ve} X^\ve(s)\Big)\dd s=\lim_{\ve  \to 0}
\int_0^t f\Big(z^\ve\big(\frac{s}{\ve^2}\big)\Big)\dd s \\
& =\lim_{\ve  \to 0}
\ve^2\int_0^{t\over \ve^2}  f\big(z^\ve(r)\big)\dd r 
=t\bar f.
\end{split}
\label{average}
\end{equation}
The convergence is in $L^p$ for every $p>0$. This applies in particular to  $H_1$ and $H_2$. 
The zero mean martingale diffusion process  $X^\ve(\cdot)-x$ has  quadratic variation 
$$[X^{\ve}-x]_t=\int_0^t \big[ H_1({1\over \ve}X^\ve(s))]^2 \dd s+\int_0^t [H_2({1\over \ve}X^\ve(s))\big] ^2 \dd s$$
which -- due to \eqref{average} -- converges in $L^1$ to $\beta_1^2 t$, where 
$$\beta_1:=\sqrt{\left(\int_0^1 \left(H_1^2(x)+H_2^2(x) \right) \dd \mu(x)\right)}=v^{-1/2}.$$
Next, we show that
\begin{equation}\label{jointqv}
[X^\ve-x,X^{\tilde \ve}-x]_t \to \hat\alpha^2 t \qquad \mbox{ for } \hat\alpha:=(\bar H_1^2 +\bar H_2^2)^{1/2} \mbox{ as } \ve,\widetilde \ve/\ve \to 0.  
\end{equation}
We have
$$
[X^\ve-x,X^{\widetilde \ve}-x]_t = \int_0^t H_1(\frac{1}{\ve} X^\ve(s))H_1({1\over {\widetilde \ve}} X^{\widetilde \ve}(s))\dd s
+\int_0^t H_2({1\over \ve}X^\ve(s))  H_2({1\over\widetilde \ve} 
X^{ \widetilde \ve}(s))\dd s.
$$
Let $f:\R \to \R$ be continuous and periodic with period 1 and $g(s):=f(s)-\bar f$. Then
\begin{align*}
\int_0^t &f\big(\frac 1\ve X^\ve(s)\big) f\big(\frac 1{\widetilde \ve} X^{\widetilde \ve}(s)\big) \dd s
=\int_0^t g\big(\frac 1\ve X^\ve(s)\big) g\big(\frac 1{\widetilde \ve} X^{\widetilde \ve}(s)\big) \dd s\\
&+\bar f^2 t + \bar f \int_0^t g\big(\frac 1\ve X^\ve(s)\big) \dd s +  
\bar f \int_0^t g\big(\frac 1{\widetilde \ve} X^{\widetilde \ve}(s)\big) \dd s.
\end{align*}
The sum of the last three terms converges to $\bar f^2 t$ by \eqref{average}, 
so in order to prove \eqref{jointqv}, it suffices to show that the first term converges to zero in probability.
Let $N:= \lceil 1/(\ve \tilde \ve) \rceil$, $b_i:=it/N, i=0,1,...,N$, and $C:= \sup_{x \in [0,1]} |g(x)|$. Then

\begin{align*}
&\Big| \int_0^t g\big(\frac 1\ve X^\ve(s)\big) g\big(\frac 1{\widetilde \ve} X^{\widetilde \ve}(s)\big) \dd s \Big|
= \Big| \sum_{i=0}^{N-1} \int_{b_i}^{b_{i+1}} g(z^\ve(s/\ve^2))g(z^{\widetilde \ve}(s/{\widetilde \ve^2}))\dd s \Big|\\
&= \widetilde \ve^2  \Big| \sum_{i=0}^{N-1} \int_{b_i/\widetilde \ve^2}^{b_{i+1}/\widetilde \ve^2 }
 g(z^\ve(s\widetilde \ve^2/\ve^2))g(z^{\widetilde \ve}(s))\dd s \Big|\\
&\le \widetilde \ve^2   \sum_{i=0}^{N-1}  \Big| g\big(z^\ve\big(\frac{b_i}{\ve^2}\big)\big)\Big|   
\Big| \int_{b_i/\widetilde \ve^2}^{b_{i+1}/\widetilde \ve^2 } g(z^{\widetilde \ve}(s))\dd s\Big|\\
&\qquad \qquad+
\widetilde \ve^2   \sum_{i=0}^{N-1} \int_{b_i/\widetilde \ve^2}^{b_{i+1}/\widetilde \ve^2 } 
\Big|g\big(z^\ve \big(\frac{\widetilde \ve^2s}{\ve^2}\big)\big)- 
g\big(z^\ve\big(\frac{b_i}{\ve^2}\big)\big) \Big|\,\Big| g(z^{\widetilde \ve}(s))\Big|\dd s\\
&\le \widetilde \ve^2C\sum_{i=0}^{N-1} \Big|\int_{b_i/\widetilde \ve^2}^{b_{i+1}/\widetilde \ve^2 } g(z^{\widetilde \ve}(s))\dd s\Big|
+ \widetilde \ve^2C\sum_{i=0}^{N-1} \int_{b_i/\widetilde \ve^2}^{b_{i+1}/\widetilde \ve^2 } 
\Big|g\big(z^\ve \big(\frac{\widetilde \ve^2s}{\ve^2}\big)\big)- g\big(z^\ve\big(\frac{b_i}{\ve^2}\big)\big)\Big| \dd s. 
\end{align*}
The expected value of the first term converges to 0 as $\ve \to 0$ by the ergodic theorem since $\tilde \ve/\ve \to 0$ and 
$\E\big|g\big(z^\ve \big(\frac{\widetilde \ve^2s}{\ve^2}\big)\big)- g\big(z^\ve\big(\frac{b_i}{\ve^2}\big)\big)\big|$ converges to zero 
as $\ve \to 0$ uniformly for all $i$, $s \in [b_i\tilde \ve^{-2},b_{i+1}\tilde \ve^{-2}]$ since $z^{\ve}$ has 
uniformly bounded volatility. This proves \eqref{jointqv}.  

All that remains to show is that $\hat\beta:=\sqrt{\beta_1^2-\hat\alpha^2}>0$ but this is true (by Jensen's inequality) since 
$\int H_2^2(x) \dd \mu(x) \ge \bar H_2^2$ and $\int H_1^2(x) \dd \mu(x) > \bar H_1^2$ since $H_1$ is non-constant.
Therefore the proof of the lemma is complete.
\end{proof}

We will need the following elementary lemmas. 

\begin{lemma}\label{independent}
Let $W,B^1,B^2,...$ be independent standard Brownian motions and let $\hat\alpha,\hat \beta,a,\delta,S,T>0$. Then
$$
\lim_{n \to \infty} \P\Big(\cup_{i=1}^n \Big( \big\{\sup_{0 \le t \le S} (\hat \beta B_t^i+\hat \alpha W_t) \ge a \big\}  \cap 
\big\{\inf_{0 \le t \le T} (\hat \beta B_t^i+\hat \alpha W_t) \ge -\delta \big\} \Big)  \Big) = 1.
$$
\end{lemma}

\begin{proof}
For $i \in \N$, let
$$
A_i:=\{\omega: \sup_{0 \le t \le S} (\hat \beta B_t^i+\hat \alpha W_t) \ge a,\; \inf_{0 \le t \le T} (\hat \beta B_t^i+
\hat \alpha W_t) \ge -\delta\}.
$$
Birkhoff's ergodic theorem implies that 
$$
\lim_{n \to \infty} \frac 1n \sum_{i=1}^{n} \1_{A_i} = \P\big( A_1|\sigma(W)\big) \mbox{ a.s.},
$$
which is strictly positive almost surely, so the assertion of the lemma follows.  
\end{proof}


The following is a quantitative version of the Borel Cantelli lemma, which provides an upper bound, and as a corollary a 
lower bound, for $M$ events   out of $N$ events to happen simultaneously. The lemma was proposed by Martin Hairer who also 
supplied an intuitive proof which our proof is based on.

\begin{lemma}[Hairer's Borel-Cantelli lemma]
\label{BC}
Let $(\Omega,\F,\P)$ be a probability space and  $\{A_i\}, 1\le i\le N$ events with $\P(A_i) = p_i$. Then 
\begin{itemize}
\item  the probability that at least $M$ of the events 
happen simultaneously is smaller or equal to $\sum_{i=1}^N p_i / M$;
\item the probability that at least $M$ of the events $\{A_i\}$
happen simultaneously is at least ${\sum_{i=1}^N p_i-M+1\over N-M+1}$.
\end{itemize}\end{lemma}

\begin{proof} Let $Q_{M,N}$ be the set of $\omega$ which belong to at least $M$ of the $N$ events from $\{A_i\}$. 
\begin{eqnarray*}
\P(Q_{M,N} )&=&\P\{\omega: \#\{1\le i\le N: \omega\in A_i\} \ge M\}\\
&=&\P\{\omega: \sum_{i=1}^N \1_{A_i}(\omega) \ge M\}
\le{1\over M}\E  \sum_{i=1}^N \1_{A_i} ={1\over M}\sum_{i=1}^N p_i.
\end{eqnarray*}

For the corresponding lower bound denote by $B_i$ the complement of $A_i$ and $q_i:=\P(B_i) = 1-p_i$.
 Let $Q^c_{M,N}$ be the complement of $B_{M,N}$, which is the event that at most $M-1$ of the events $A_i$ happen
 or --  equivalently -- the set on which at least $N-M+1$ events from the $\{B_i: 1\le i\le N\}$ happen. 
It follows from the previous lemma that
 \begin{equation}\nonumber
\P(Q^c_{M,N}) \le {\sum_{i=1}^N q_i \over N-M+1} = {N - \sum_{i=1}^N p_i \over N-M+1}\;,
\end{equation} 
so that 
\begin{equation}\nonumber
\P(Q_{M,N}) \ge 1- {N - \sum_{i=1}^N p_i \over N-M+1} = {\sum_{i=1}^N p_i - M + 1 \over N-M+1}\;,
\end{equation}
as required.
 \end{proof}

\begin{corollary}
\label{corollary2}
Let $0<\alpha \le \beta$. Then, for every $T>0$, $\ve>0$,  
there exists a $\delta>0$, such that for each $m \in \N$, there exists some $N \in \N$
such that the following holds: 
for every sequence $M_1, M_2,\dots$ of martingales with continuous paths on the same space 
$(\Omega,\F,\P)$, starting at zero 
such that $\alpha \le \frac \dd {\dd t} \langle M_i\rangle_t\le \beta $ for all $i$ and $t$,
the stopping time 
$$
\tau:=\inf\big\{t>0: M_i(t) \ge \delta \mbox{ for at least } m \mbox{ different } i \in \{1,...,N\}\big\}
$$
satisfies
$$
\P\{\tau \le T\} \ge 1-\ve.
$$
\end{corollary}

\begin{proof} 
For $\delta> 0$, let $\lambda_i^{\delta}:= \mu\{0 \le t \le T: M_i(t) \ge \delta\}$, where $\mu$ denotes normalized 
Lebesgue measure on $[0,T]$. 
We claim that there exist $\delta >0$ and $u>0$ such that for all $i \in \N$, we have
\begin{equation}\label{Lebesgue}
\P\big\{\lambda_i^{\delta} \ge u \big\}\ge 1 - \frac {\ve}{2}. 
\end{equation}
Assume that this has been shown. For $k \ge 2$, let 
$$
\Omega_k:=\{\lambda_i^{\delta} \ge u \mbox{ for at least } k \mbox{ different } i \in \{1,...,2(k-1)\} \}.
$$
Then, by Lemma \ref{BC},
$$
\P\big(\Omega_k\big) \ge \frac{2(k-1)(1-\frac \ve 2) - k +1}{2(k-1)-k+1} = 1-\ve.
$$
Invoking Lemma \ref{BC} once more, we see that on $\Omega_k$, there exists some $t \in [0,T]$ such that 
$M_i(t) \ge \delta$ for at least $m$ different $i \in \{1,...,2k-2\}$ provided that the numerator $uk-m+1$ in the formula in 
Corollary  \ref{BC} is strictly positive. Letting $k:=\lceil \frac mu \rceil$ and $N:=2k-2$, 
the assertion of the corollary follows at once. 

It remains to prove \eqref{Lebesgue}. Let $\delta>0$ (we will fix the precise values later). 
For ease of notation, we drop the index $i$ (observe that all estimates below are uniform in $i$).
The martingale $M$ can be represented as a time-changed Brownian motion: $M(t)=W([M]_t)$. Since 
$\frac \dd {\dd t} [M]_t\in [\alpha,\beta]$, we obtain for $a,t>0$
\begin{eqnarray*}
\P(\sup_{s\in [0,t]} M(s)\ge a)&=&\P(\sup_{0\le s\le t} W([M]_s) \ge a)\\
&\ge& \P(\sup_{0\le s\le \alpha t} W(s) \ge a)\\
&=& 2\,\P(W(\alpha t) \ge a) = {2\over \sqrt{2 \pi}}\int_{a/\sqrt{\alpha t}}^\infty \ee^{-{x^2\over 2}} \dd x=:p_0(a,t),
\end{eqnarray*}
and, analogously,
$$
\P(\inf_{s\in [0,t]} M(s)\ge -a)\ge \P(\sup_{0 \le s \le t}W(\beta s) \le a) 
= {2\over \sqrt{2 \pi}}\int_0^{a/\sqrt{\beta t}} \ee^{-{x^2\over 2}} \dd x=:q_0(a,t).
$$
Let $\tilde \tau$ be the first time that $M(t)\ge 2 \delta$. Using the fact that $M(\tilde \tau \wedge c +t)-M(\tilde \tau \wedge c)$ 
also satisfies the assumptions of the corollary for each $c \ge 0$, we get for $u \in (0,\frac 12]$
\begin{eqnarray*}
&&\P\{\lambda^\delta(\omega)\ge u\}\\
&&\ge \P\left(\inf_{\tilde \tau\le t\le \tilde \tau+uT}M(t)\ge \delta \;\Big |\:
\tilde \tau\le {T\over 2}\right)\P\left(\tilde \tau\le {T\over 2}\right)\\
&&\ge q_0(\delta,uT)\, p_0(2\delta,\frac T2).
\end{eqnarray*}
Choosing first $\delta>0$ so small that the second factor is close to $1$ and then choosing  $u>0$ small enough, 
we can ensure that the product is at least $1-\ve/2$ proving \eqref{Lebesgue}, so the proof of the corollary 
is complete. 
\end{proof}

\subsection{The Examples}

\begin{example}
Consider the SDE \eqref{SDE1}. We will start by defining the coefficient $\sigma$ restricted to $\R \times [0,1]$. 
Fix a smooth non-constant, strictly positive 
function $H$ of period one. To construct the example, we subdivide 
the square $[n,n+1]\times [0,1]$ into $M_n$ horizontal strips of width $1/M_n$ each, with $M_n$ 
increasing sufficiently quickly and 
let $\sigma$ be equal to $H$ sped up by a factor depending on the particular strip. 
Thus, the probability that one of the solutions starting from $(n,y)$, $y \in [0,1]$ will reach the next level $(n+1,y)$ within 
a very short time will increase with $n$ allowing us to conclude that strong completeness fails. 
We now state the precise assumptions.

Let $H: \R\to [\frac 12,1]$  be an infinitely differentiable non-constant  
function with period $1$. Assume that all its derivatives vanish at 0. 
Fix a sequence of positive integers $a_i$, $i=0,1,...$ such that $a_0=1$ and $\lim_{i \to \infty} a_{i+1}/a_i = \infty$. 
Assume that $N_0,N_1,N_2,...$ are positive even integers whose values we will fix later. Let 
$M_n:=\prod_{i=0}^n N_i$, $n \in \N_0$ and define
\begin{align}\label{sigma}
\sigma(x,y)= H(a_i x)  \mbox{ if } &i \in \{0,1,...,\frac{N_n}{2}-1\},\;  x \in [n,n+1],\\
&y \in \left[\frac{kN_n+2i}{M_n},\frac{kN_n+ 2i+1}{M_n}\right],\,k=0,...,M_{n-1}-1.\nonumber 
\end{align}
Further, let $\sigma(x,y)=H(x)$ for $x\le 0, y \in [0,1]$. 
On the set where $\sigma$ has been defined, it is clearly bounded, strictly positive, and $C^{\infty}$ 
(since we assumed that all derivatives of $H$ to vanish at zero). 
It is also clear that $\sigma$ 
can be extended to a  $C^{\infty}$ function taking values in $[1/2,1]$ on all of $\R^2$.
We claim that the associated flow is not strongly complete in case 
the integers $N_0, N_1, ...$ are chosen to increase sufficiently quickly. 

Let $\psi_{st}$ denote the $x$-component of the maximal flow $\phi$ of the SDE started at time $s$ ($s\le t$). 
We define a sequence of stopping times $\tau_n$, $n \in \N_0$ 
and intervals $I_n \subseteq [0,1]$ as follows: $\tau_0:=0$, $I_0:=[0,1]$,  
$\tau_{n+1}:= \inf \{t > \tau_n: \sup_{y \in I_n} \psi_{\tau_n t}(n,y)=n+1\}$ and let $I_{n+1} \subseteq I_n$ be some interval 
of the form $\left[\frac {2k}{M_n},\frac{2k+1}{M_n}\right]$ on which the supremum in the definition of $\tau_{n+1}$ is attained 
(note that the supremum is attained for {\em every} point in such an interval if it is attained 
for {\em some} point in the interval). 
Define 
$\tau:=\inf \{t \ge 0: \sup_{y \in [0,1]}\psi_{0 t}(0,y)=\infty\}$. Then $\tau \le \lim_{n \to \infty} \tau_n$ and
it suffices show that $\P\{\tau_{n+1}-\tau_n \ge 2^{-n}\}$ 
is summable over $n$ to deduce that $\P\{\tau < \infty\}>0$ (a straightforward argument then shows that even $\tau < \infty$ 
almost surely).  

Fix $n \in \N$. We will show that we can choose $N_n \in \N$ in such a way that 
$$
\P\{\tau_{n+1}-\tau_n \ge 2^{-n}|\F_{\tau_n}\} \le 2^{-n}.
$$ 

Let $\hat y \in I_n$ and let $M_j^n$ solve the following SDE:
\begin{align*}
\dd M_{j}^n(t)&=\xi_{j}^n(M_{j}^n(t)) \dd W(t),\\ 
 M_{j}^n(0)&=n,
\end{align*}
where
$$
\xi_{j}^n(z):=\left\{ 
\begin{array}{ll}
\sigma(z,\hat y)\; &{\rm if } \, z \le n\\
H(a_j\,z)\; &{\rm if } \, z \ge n.
\end{array}
\right.
$$
Observe that $\xi_j^n$ does not depend on the particular choice of $\hat y \in I_n$ and that 
(up to a shift of the Wiener process $W$) $M_j^n(t)$, $j=0,..,\frac{N_n}{2} -1$ are the solutions of our SDE 
after $\tau_n$ and until $\tau_{n+1}$ on the intervals $ \left[\frac{lN_n+2j}{M_n},\frac{lN_n+ 2j+1}{M_n}\right]$, where     
$l$ is chosen such that $(lN_n+1)/M_n \in I_n$. We need to ensure that for $N_n$ large enough, one of the $M_j^n$ will reach 
the next level $n+1$ within time $2^{-n}$ with probability at least $1-2^{-n}$. Unfortunately, we cannot apply the homogenization 
lemma \ref{homo} directly to the $M_j^n$, since they all have the same diffusion coefficient for $z \le n$. 
Therefore, we will wait at most time $T_n=\frac 12 2^{-n}$ and show that for $N_n$ large, it is very likely, that 
many of the  $M_j^n$ have reached at least level $n+\delta_n$ for some (possibly very small) $\delta_n>0$. We will then apply 
the homogenization lemma only to these $M_j^n$. Of course, this is possible only if the solution does not go back to level $n$ before 
time $\tau_{n+1}$. Lemma \ref{independent} ensures, that with high probability, we can find at least one of the 
remaining $M_j^n$ for which this is true. We now provide the details of the argument.\\

\noindent {\bf Step 1:} We apply Corollary \ref{corollary2} to the martingales $M_j^n-n$, $j=0,1,2,...$  with $T_n=\frac 12 2^{-n}$, 
$\ve_n=\frac 14 2^{-n}$, $\alpha=1/4$, $\beta=1$ and obtain a number $\delta_n>0$ which satisfies (\ref{Lebesgue}) 
in the proof of Corollary \ref{corollary2}. 
We can assume that $\delta_n<1$. \\

\noindent {\bf Step 2:} Now, we define $\tilde M_j^n$, $j \in \N_0$ as the solution of the SDE
\begin{align*}
\dd \tilde M_{j}^n(t)&=H(a_j\,\tilde M_{j}^n(t)) \dd W(t),\\ 
\tilde M_{j}^n(0)&=n+\delta_n.
\end{align*}
Applying Lemma \ref{homo} to $\tilde M_0^n,\tilde M_1^n,...$ with $x=n+\delta_n$, $H_1=H$, $H_2=0$, we see that 
$(\tilde M_k^n-x,\tilde M_{k+1}^n-x,...)$ converges in law to 
$(\hat \alpha B_0+\hat \beta B_1,\, \hat \alpha B_0+\hat \beta B_2,...)$ as $k \to \infty$, 
where $\hat \alpha,\,\hat \beta >0$ and 
$B_0,B_1,...$ are independent standard Wiener processes.\\

\noindent {\bf Step 3:} Next, Lemma \ref{independent} says that there exists some $m_n \in \N$ such that
$$
\P\Big(\bigcup_{i=1}^{m_n} \Big( \big\{\sup_{0 \le t \le \frac 12 2^{-n}} (\hat \alpha B_0(t)+\hat \beta B_i(t)) \ge 1 \big\}  \cap 
\big\{\inf_{0 \le t \le 2^{-n}} (\hat \alpha B_0(t)+\hat \beta B_i(t)) \ge -\frac{\delta_n}{2} \big\} \Big)  \Big) \ge 1-\ve_n.
$$\\

\noindent {\bf Step 4:} Let $\tilde N_n$ be the number in the conclusion of Corollary \ref{corollary2} associated 
to $\delta_n,\ve_n,T_n,$ and $m_n$. 
Thanks to the convergence stated in Step 2, we can find some $k_n \in \N$ such that
for any subset $J \subseteq \{k_n,k_n+1,...,k_n+\tilde N_n-1\}$ of cardinality $m_n$, we have
$$
\P\Big(\bigcup_{i \in J} \Big( \big\{\sup_{0 \le t \le \frac 12 2^{-n}} \tilde M_i^n \ge n+1 \big\}  \cap 
\big\{\inf_{0 \le t \le 2^{-n}} \tilde M_i^n \ge -{\delta_n} \big\} \Big)  \Big) \ge 1-\ve_n.
$$\\

\noindent {\bf Step 5:} Define $N_n:= k_n+\tilde N_n$. Using the strong Markov property and the fact that $\psi$ is 
order preserving (and hence a solution starting at $n+\delta_n$ can never pass a solution starting at a larger value at 
the same time), we obtain for our choice of $N_n$ that
$$
\P\big\{ \tau_{n+1}-\tau_n \ge 2^{-n}\big|\F_{\tau_n}\big\}\le 3\ve_n<2^{-n}
$$
as desired, so the proof is complete.
\hfill $\Box$
\end{example}

Note that if the SDE in the above example is changed into Stratonovich the SDE is strongly complete.
To produces an example in Stratonovitch form, we use two independent Brownian motions.

\begin{example}
Consider the SDE 
\begin{eqnarray}\label{sde1}
\begin{split}
\dd X(t)&=\sigma_1(X(t),Y(t))\dd W_1(t) + \sigma_2(X(t),Y(t))\dd  W_2(t) \\ 
\dd Y(t)&=0,
\end{split}
\end{eqnarray} 
where $W_1,\, W_2$ are two independent standard one-dimensional Brownian motions.
We will construct bounded and $C^{\infty}$ functions $\sigma_1$ and $\sigma_2$ such that 
$\sigma_1^2(x,y)+\sigma_2^2(x,y)=1$ for all $x,y$  such that the associated flow is 
not strongly complete. Note that due to the condition $\sigma_1^2(x,y)+\sigma_2^2(x,y)=1$,
it does not matter if we interpret the stochastic differentials in the It\^o or 
Stratonovich sense.

The construction of the example resembles that of the previous one closely, the only difference being that 
this time, we consider two non-constant $C^{\infty}$ functions $H_1$, $H_2$ taking values in $[1/2,1]$ such 
that $H_1^2(z)+H_2^2(z)=1$ and apply  
Lemma \ref{homo} with these functions $H_1,H_2$ rather than with a single function $H$ as before.\hfill $\Box$
\end{example}

We essentially showed that  we could trace back to and construct a random initial point $x_0(\omega)$ 
which goes out fast enough to explode. This is   true in general: Suppose that there is a maximal flow 
$\{\phi_t(x,\omega), t<\tau(x, \omega)\}$ to the SDE. It is strongly complete if and only if for all 
measurable random points $x(\omega)$ on the state space,  
$\phi_t(x(\omega), \omega)$ exists almost surely for all $t$.

\begin{remark}
It remains an open question whether an SDE with globally Lipschitz diffusion coefficients and 
a drift which is locally Lipschitz and of linear growth admits a global solution flow. 
\end{remark}

\end{document}